\documentclass[leqno]{amsart}

\usepackage[utf8]{inputenc}
\usepackage[T1]{fontenc}
\usepackage[english]{babel}
\usepackage{amsmath}
\usepackage{amsfonts}
\usepackage{amssymb}
\usepackage{enumitem}
\usepackage{mathabx}
\usepackage{color}
\usepackage{tikz}
\usepackage{pgfplots}
\usetikzlibrary{pgfplots.groupplots}

\def\eps{\varepsilon }
\newcommand\cU{{\mathcal U}}

\newcommand\cQ{{\mathcal Q}}
\newcommand\cD{{\mathcal D}}
\newcommand\cS{{\mathcal S}}
\newcommand\cV{{\mathcal V}}
\newcommand\cM{{\mathcal M}}
\newcommand\tcV{\widetilde{\mathcal V}}
\newcommand\ha{{\widehat{a}}}
\newcommand\hb{{\widehat{b}}}
\newcommand\hc{{\widehat{c}}}
\newcommand\hd{{\widehat{d}}}
\newcommand\hq{{\widehat{q}}}

\newcommand\ta{{\widetilde{a}}}
\newcommand\tb{{\widetilde{b}}}
\newcommand\tc{{\widetilde{c}}}
\newcommand\td{{\widetilde{d}}}

\newcommand\nrm{\vvvert}
\newcommand\dD{\mathrm{d}}

\newcommand\lla{\left\langle}
\newcommand\rra{\right\rangle}

\newcommand\RR{{\mathbb{R}}}
\newcommand\NN{{\mathbb{N}}}
\newcommand\ZZ{{\mathbb{Z}}}

\newcommand\dx{\mathrm{d}x}
\newcommand{\be}{\begin{equation}}
  \newcommand{\ee}{\end{equation}}
\newcommand{\bt}{\begin{theorem}}
  \newcommand{\et}{\end{theorem}}
\newcommand\ds{\displaystyle}
\newcommand{\bl}{\begin{lemma}}
  \newcommand{\el}{\end{lemma}}
\newcommand{\br}{\begin{rema}}
  \newcommand{\er}{\end{rema}}
\newcommand{\bpr}{\begin{proposition}}
  \newcommand{\epr}{\end{proposition}}
\newcommand{\bc}{\begin{corollary}}
  \newcommand{\ec}{\end{corollary}}
\newcommand{\bdf}{\begin{defi}}
  \newcommand{\edf}{\end{defi}}

\newtheorem{theorem}{Theorem}[section]
\newtheorem{proposition}[theorem]{Proposition}
\newtheorem{corollary}[theorem]{Corollary}
\newtheorem{lemma}[theorem]{Lemma}
\newtheorem{rema}[theorem]{Remark}
\newtheorem{defi}[theorem]{Definition}

\usepackage[hidelinks]{hyperref}

\usepackage{graphicx}

\title[Projection on polynomials with two bounds]{A projection algorithm on the set of polynomials with two bounds}

\author{M. Campos Pinto}
\author{F. Charles}
\author{B. Despr\'es}
\address[M. Campos Pinto, F. Charles, and B. Despr\'es]{LJLL,  Sorbonne Universit\'e,  CNRS UMR 7598, F-75005, Paris, France}
\email[M. Campos Pinto]{campos@ljll.math.upmc.fr}
\email[F. Charles]{frederique.charles@ljll.math.upmc.fr}
\email[B. Despr\'es]{bruno.despres@sorbonne-universite.fr}
\author{M. Herda}
\address[M. Herda]{Inria, Univ. Lille, CNRS, UMR 8524 – Laboratoire Paul Painlevé, F-59000 Lille, France.}
\email[M. Herda]{maxime.herda@inria.fr}

\thanks{MH acknowledges support by the Labex CEMPI (ANR-11-LABX-0007-01) and the Labex SMP (ANR-10-LABX-0098).}

\begin{document}

\begin{abstract}
 The motivation of this work  stems  from the numerical approximation of bounded functions by polynomials satisfying the same bounds.
 The present contribution makes use of the recent algebraic characterization  found in [B. Després, \textit{Numer. Algorithms}, 76(3), (2017)] and
 [B. Després and M. Herda, \textit{Numer. Algorithms}, 77(1), (2018)] where 
an  interpretation of monovariate polynomials  with two bounds is provided in terms of
a  quaternion algebra  and the Euler four-squares formulas.
 Thanks to this structure, we generate a new nonlinear projection algorithm onto the set of polynomials with two bounds. The numerical analysis of the method provides theoretical error estimates showing stability and continuity of the projection. Some numerical tests illustrate this novel algorithm for constrained polynomial approximation.
  
  \bigskip
  
  
  \medskip

  	\noindent\textsc{MSC2010 subject classification.}  65D15, 41A29, 90C20.
  \end{abstract}
  \maketitle

  \section{Introduction}

  Given $n\in \NN$ we let $P_n$ be the set of univariate polynomials of degree less or equal to $n$, and
  set by convention $P_{-1} = \{0\}$. 
  A central result is the Luk\`acs Theorem \cite[Sec.~1.21]{szego_1975_orthogonal} 
  which characterizes polynomials with one lower bound.
  Specifically, let $P_n^+\subset P_n$ be the subset of positive (or nonnegative) polynomials on the 
  segment $[0, 1]$, namely  
  $$
  P_n^+\ :=\ \{p\in P_n,\ \text{ such that }\  0\leq p(x) \ \text{ for all }\ x\in[0,1]\}.
  $$
  In this article we will consider the case of even degrees.
  The extension to odd degrees is essentially a question of technical matters, with no new
  ideas with respect to the material presented in this work.

  \bt[Even degree \cite{szego_1975_orthogonal}]   \label{t:lukacs}
  Take $n\in\NN$ and  $p\in P_{2n}^+$.  Then there exists
  $a\in P_n$ and $b\in P_{n-1}$ such that
  $
  p(x) =a^2(x) + b^2(x)w(x)$
  with weight $w(x) = x(1-x)$.
  \et

   The problem considered in this article is the design and analysis of a nonlinear projection algorithm onto the set of polynomials with one lower bound and one upper bound,
  \[
  U_{2n}\ :=\ \{p\in P_n,\ \text{ such that }\  0 \leq p(x) \leq 1 \ \text{ for all }\ x\in[0,1]\}.
  \]
  Our approach is based on the observation that we have
  $$
  U_{2n} = \{p\in P_{2n}^+ \mid 1-p \in P_{2n}^+\}.
  $$
  Given that there already exist projection algorithms on 
  $P_{2n}^+$ (see \cite{campos_2019_algorithms,despres_2018_iterative}),
  our present objective is to design a nonlinear algorithm that maps a pair 
  $(p_0,1-p_1)\in P_{2n}^+\times P_{2n}^+$ into $U_{2n}$.
  To do so we will first describe a specific parametrization of the set $U_{2n}$ that heavily
  relies on the four-squares identity of Euler \cite[p.~54]{euler_1760_demonstratio}.
  This theoretical framework will then be used to build a practical algorithm 
  for bounded polynomial approximation. 
  To our knowledge, this work is the first attempt to use the algebraic structure of Euler's
  identity to build an algorithm with such advanced properties.

  The organization is as follows.
  In the next section we introduce some elementary concepts and notation, and we specify some of 
  the aforementioned algebraic properties: the quaternion structure is recalled, its expression 
  in the Chebychev basis is given and some norms are defined.
  In Section~\ref{s:projection} we then specify our approximation problem with two bounds and
  define the nonlinear projection algorithm: it is an extension of the theoretical decomposition
  method from \cite{despres_2016_polynomials_erratum, despres_2016_polynomials} with a new 
  nonlinear correction step.
  In Section~\ref{s:analysis} we perform the numerical analysis of the method and state
  in   Theorem~\ref{t:error2}
   a continuity or stability result. Finally in the last section we illustrate the method with some simple 
  numerical tests.
   
  \section{Notations and basic properties of $U_{2n}$}\label{s:notation}
  
    \subsection{Representation of polynomials with two bounds}
  A polynomial $p$ belongs to $U_{2n}$ if and only if $p\in P_{2n}^+$ and $1-p\in P_{2n}^+$. 
  Define the set of quadruplets
  \[
  \cQ_n\ :=\ P_n\times P_{n-1}\times P_n\times P_{n-1}.
  \]
 By Theorem~\ref{t:lukacs}, for any $p\in U_{2n}$, there is a quadruplet $q = (a,b,c,d)\in \cQ_n$ such that $a^2(x) \,+\, b^2(x)\,w(x)\,+\,c^2(x) \,+\, d^2(x)\,w(x) = 1$. 
 It is convenient to define the function $M: \cQ_n \rightarrow  P_{2n}$ by 
  \begin{equation}
  M(q)(x)\ :=\ a^2(x) \,+\, b^2(x)\,w(x)\,+\,c^2(x) \,+\, d^2(x)\,w(x)\,,
  \label{eq:modulus}
  \end{equation}
  and  the set 
  \[
  \cU_{n} = \{q\in\cQ_n,\ \text{ such that }\  M(q) = 1\}\,.
  \]
  The function $(a,b,c,d)\mapsto a^2 + b^2w$ maps $\cU_n$ onto $U_{2n}$,
  so it is sufficient to characterize $\cU_n$ to get a characterization of 
  the set of polynomials $U_{2n}$.
  
  A central tool will be a recent factorization result recalled in Theorem~\ref{th:decomp} below,
  that involves a multiplication law on quadruplets based
  on Euler's four-square identity \cite{euler_1760_demonstratio}.
  Given two elements $r = (\alpha,\beta,\gamma,\delta)$ and $q = (a,b,c,d)$ in 
  $\cQ_{\infty} = \cup_{n\in\NN}\cQ_n$,
  we define $rq := (A,B,C,D) \in \cQ_{\infty}$ with
  \be
  \left\{
  \begin{array}{ccccccccc}
    A &=& \alpha\,a &-&\beta\,b\,w &-&\gamma\,c&-&\delta\,d\,w\,,\\[.5em]
    B &=& \beta\,a &+&\alpha\,b &-&\delta\,c&+&\gamma\,d\,,\\[.5em]
    C &=& \gamma\,a &+&\delta\,b\,w &+&\alpha\,c&-&\beta\,d\,w\,,\\[.5em]
    D &=& \delta\,a &-&\gamma\,b &+&\beta\,c&+&\alpha\,d\,.
  \end{array}
  \right.
  \label{e:foursquare}
  \ee
  Note that this is actually a modified version of Euler's four-square identity, 
  where the signs are different. The sign convention adopted here will make it simpler
  to describe $\cQ_{\infty}$ by quaternions.
  The neutral element for this multiplication law is $(1,0,0,0)$, and 
   every element of $\cU_{\infty} = \cup_{n\in\NN}\cU_n$ has an inverse.
  Indeed, define the conjugate of 
  $q = (a,b,c,d)$ in $\cQ_{\infty}$ by
  \[
  \overline{q}\ =\ (a,-b,-c,-d).
  \]
  Then a direct application of formula \eqref{e:foursquare} yields
  \[
  \overline{q}q\ =\ q\overline{q}\ =\  (M(q), 0, 0, 0), \qquad \forall\ q \in \cQ_{\infty}.
  \]
  In particular,
  \[
  \overline{q}q\ =\ q\overline{q}\ =\  (1, 0, 0, 0), \qquad \forall\ q \in \cU_{\infty}
  \]
  so that $\cU_{\infty}$ has indeed a non-commutative group structure.  
  Note that $\overline{\overline{q}} = q$ and that $\overline{r\,q} = \overline{q}\,\overline{r}$. 
  Moreover $M$ is a morphism, namely $M(qr) = M(q)M(r)$ for any quadruplets 
  $q$ and $r$ in $\cQ_{\infty}$.  
  With an additional natural addition defined by $(\alpha,\beta,\gamma,\delta) + (a,b,c,d) = (\alpha + a, \beta + b, \gamma + c, \delta + d)$ and a scalar multiplication 
  $\lambda(a,b,c,d) = (\lambda a,\lambda b,\lambda c,\lambda d)$, $\cQ_{\infty}$ is a 
  non-commutative $\RR$-algebra which inherits all its algebraic properties from the quaternions. 
  Indeed if one represents the quadruplet $(a,b,c,d)$ by the following 
  quaternion-valued formal function $a + ib\sqrt{w} + jc + kd\sqrt{w}$, 
  then the usual quaternions operations based on the relations $i^2 \,=\, j^2 \,=\, k^2 \,=\, ijk \,=\, -1$
  coincide with those introduced here on our polynomial quadruplets. 
  In this sense, the equality holds
  $$
  (a,b,c,d)= a + ib\sqrt{w} + jc + kd\sqrt{w} \in \mathcal Q_n.
  $$
  The interest of this algebraic formalism lies in the following factorization result.
  
  \bt[\cite{despres_2016_polynomials_erratum,despres_2016_polynomials}]\label{th:decomp}
  Let $n\in\NN$. For any $q\in\cU_n$ there is $e\in\cU_1$ such that $eq \in \cU_{n-1}$. As a consequence, any quadruplet  $q\in\cU_n$ admits a factorization in at most $n$ elements $e_1,\,e_2,\,\dots,\,e_n$
  of $\cU_1$
  \begin{equation}
  q\ =\ e_1\,e_2\,\dots\,e_n\,.
  \label{eq:decomp}
  \end{equation}
    \et
  
   The structure of the proof \cite{despres_2016_polynomials_erratum, despres_2016_polynomials}  is as follows. One starts from $q\in \cU_n$ and shows that there exists $e_1\in \cU_1$ such that $\overline{e_1}q\in \cU_{n-1}$. The construction of $e_1$ is explicit and based on the examination of the two dominant coefficients of each of the four polynomial components of $q$. The proof is ended by iteration on $n$, $n-1$, \dots
  
  On the basis of this result, one has a constructive characterization of polynomials with bounds. The question addressed in the present 
  work  is the evaluation of this structure for algorithmic purposes. Since \eqref{eq:decomp} is a very 
  nonlinear formula, it is not easy to handle. However, in the rest of this article, we will show that it is possible to obtain an efficient  nonlinear projection onto $\cU_n$ using this structure. 
  
  \subsection{Chebychev basis}
It is well-known that Chebychev polynomials enjoy good stability properties which make them
suitable for numerical algorithms \cite{trefethen_2010_chebfun}.
Indeed some preliminary tests \cite{despres_2016_polynomials} for the application of Theorem 
\ref{th:decomp} have confirmed that the monomial basis may suffer from very poor numerical accuracy 
for high order polynomials. 
Our findings are also that Chebychev bases are  well adapted to the expression of 
Euler's four-square formula \eqref{e:foursquare} along their coefficients. 
These reasons explain why only Chebychev bases are considered in this work for algorithmic purposes.

The shifted Chebychev polynomials of the first kind are the only polynomials such that 
  \[
  T_n\left(\frac{\cos(\theta)+1}{2}\right) = \cos(n\theta), \qquad \theta\in \mathbb R, \ n\in \mathbb N.
  \]
  The polynomial $T_n$ is of degree $n$ and
  the definition actually extends to negative indices, 
  as $T_{-n} = T_{n}$. 
  The shifted Chebychev polynomials of the second kind are the only polynomials such that 
  \[
  U_n\left(\frac{\cos(\theta)+1}{2}\right) = \frac{2\,\sin(n\theta)}{\sin(\theta)}, \qquad \theta\in \mathbb R\setminus\pi\mathbb Z, \ n\in \mathbb N.
  \]
  Now $U_n$ is of degree $n-1$, recalling our convention
  $P_{-1} = \{0\}$, and again we may extend the definition to negative indices: one has $U_{-n} = -U_n$.
  Note that the shifted Chebychev polynomials of second kind are usually defined without the factor $2$,
  and with an index that is the degree of the polynomial. Our notation will allow to simplify
  some of the subsequent computations.
   
  Chebychev polynomials enjoy natural orthogonality properties \cite{abramowitz_1964_handbook,szego_1975_orthogonal}. Define the scalar products
  \[
  \lla f,\,g\rra_T\ =\  \int_0^1 f(x)\,g(x)\,{w(x)}^{-1/2}\,\dx\,,\qquad \lla f,g\rra_U\ =\  \int_0^1 f(x)\,g(x)\,{w(x)}^{1/2}\,\dx.
  \]
  Then for any $(i,j)\in\ZZ^2\setminus\{(0,0)\}$, one has $\lla T_i,\,T_j\rra_T \ =\ \lla U_i,\,U_j\rra_U\ =\   \frac{\pi}{2}\, \delta_{ij}$ where $\delta_{ij}$ is the Kronecker symbol and $\lla T_0,\,T_0\rra_T\ =\  \pi$. These formulas are 
  established by noticing that the weight is such that
  $
  w\big(\frac{\cos(\theta)+1}{2}\big) =  \frac{\sin^2(\theta)}{4}$. 
  
  \begin{rema}
   One has the  identity for all $n\in \mathbb N$
   \[
1=   T_n(x)^2 + U_n(x)^2w(x).
   \]
  It  underlines that the Luk\`acs decomposition of a polynomial is non  unique. 
  \end{rema}

  \subsubsection{A Chebychev basis for the set $\cU_n$}
  
  Any $(a,b,c,d)\in\cU_n$ admits a Chebychev representation 
  \be \label{e:chebas}
  \begin{array}{rclrcl}
    a(x)&=&\ds\sum_{i=0}^n a_i\, T_i(x)\,,\qquad& c(x)&=&\ds\sum_{i=0}^n c_i\, T_i(x)\,,\\[1em]
    b(x)&=&\ds\sum_{i=1}^n b_i\, U_i(x)\,,\qquad& d(x)&=&\ds\sum_{i=1}^n d_i\, U_i(x)\,,
  \end{array}
  \ee
  with
  \be \label{eq:demoi1}
  \begin{array}{lll}
    a_i = \ds\frac{2-\delta_{i0}}{\pi}\lla a(x), T_i(x)\rra_T,& c_i = \ds\frac{2-\delta_{i0}}{\pi}\lla c(x), T_i(x)\rra_T,
    & i\in\{0,1,\dots,n\},
  \end{array}
  \ee
   and
  \be \label{eq:demoi2}
  \begin{array}{lll}
    b_i = \ds\frac{2}{\pi}\lla b(x), U_i(x)\rra_U, & d_i =\ds \frac{2}{\pi}\lla d(x), U_i(x)\rra_U,
    &i\in\{1,\dots,n\}.
  \end{array}
  \ee
   It will be convenient to extend these coefficients for all $i \in \ZZ$, setting 
   $a_i = c_i = 0$ or $b_i = d_i = 0$ when $i$ is outside of the above ranges. 
    
    The coefficients of the product \eqref{e:foursquare} of  two quadruplets $r$ and $q$ 
    can be expressed quite handily in the  Chebychev 
    basis. Indeed, as a consequence of the De Moivre formulas, for any $(i,j)\in\ZZ^2$ one has
  \be \label{eq:demoi}
  T_iT_j = \frac{T_{i-j} + T_{i+j}}{2},\quad U_iU_jw = \frac{T_{i-j} - T_{i+j}}{2},\quad U_iT_j =\frac{U_{i+j} + U_{i-j}}{2}.
  \ee
It is useful to consider the sign function 
$\mathrm{sgn}(x)=1$ for $x>0$, $\mathrm{sgn}(x)=-1$ for $x<0$ and $
\mathrm{sgn}(0)=0$.

\begin{lemma}
The coefficients of the polynomials  in \eqref{e:foursquare} can be expressed as 
  \be
  \left\{
  \begin{array}{rcl}
    2\,A_k &=&\ds\sum_{i+j=k}(\alpha_i\,a_j + \beta_i\,b_j - \gamma_i\,c_j + \delta_i\,d_j)\ +\ \sum_{|i-j|=k}(\alpha_i\,a_j - \beta_i\,b_j - \gamma_i\,c_j - \delta_i\,d_j) , \\[.5em]
    2\,B_k &=&\ds\sum_{i+j=k}(\beta_i\,a_j + \alpha_i\,b_j - \delta_i\,c_j + \gamma_i\,d_j)\ +\ \sum_{|i-j|=k}(\beta_i\,a_j + s_{ij}\,\alpha_i\,b_j - \delta_i\,c_j + s_{ij}\,\gamma_i\,d_j), \\[.5em]
    2\,C_k &=&\ds\sum_{i+j=k}(\gamma_i\,a_j - \delta_i\,b_j + \alpha_i\,c_j + \beta_i\,d_j)\ +\ \sum_{|i-j|=k}(\gamma_i\,a_j + \delta_i\,b_j + \alpha_i\,c_j - \beta_i\,d_j), \\[.5em]
    2\,D_k &=&\ds\sum_{i+j=k}(\delta_i\,a_j - \gamma_i\,b_j + \beta_i\,c_j + \alpha_i\,d_j)\ +\ \sum_{|i-j|=k}(\delta_i\,a_j - s_{ij}\,\gamma_i\,b_j + \beta_i\,c_j + s_{ij}\,\alpha_i\,d_j),
  \end{array}
  \right.
  \label{e:foursquare_coeff}
  \ee
  where  $s_{ij} = \mathrm{sgn}(j-i)$.
    \end{lemma}
    \begin{proof}
    Obtained from \eqref{e:foursquare} and    the De Moivre  formulas (\ref{eq:demoi}).
     \end{proof}
   
   \begin{lemma}
  Take    $q\in\cQ_n$. One can write 
  $
   M(q) = \sum_{i=0}^{2n}M(q)_i T_i$ which is  expressed  with the  Chebychev basis of the first kind only. The  dominant coefficient is 
  \be  \label{e:2n}
  M(q)_{2n}\ =\ \frac12 (a_n^2 \,-\, b_n^2 \,+\, c_n^2 \,-\, d_n^2)
  \ee
  and the next one is 
  \be  \label{e:2n-1}
  M(q)_{2n-1}\ =\  
  \left\{
  \begin{array}{ll}
    a_n\,a_{n-1}\, -\, b_n\,b_{n-1}\, +\, c_n\,c_{n-1}\, -\, d_n\,d_{n-1}&\text{ if }\ n\geq 2\,,\\[.75em]
  2\,a_1\,a_{0}\, +\, 2\,c_1\,c_{0}& \text{ if }\ n = 1\,.
\end{array}
\right.
  \ee
  \end{lemma}
  \begin{proof}
  The expansion of $M(q)$ along the Chebyshev basis shows products $T_\alpha T_\beta$ and products $U_\alpha U_\beta$.
  The De Moivre formulas  (\ref{eq:demoi}) yield that all products can be expanded along the $T_\gamma$ solely.
  Direct computations yield the coefficients $ M(q)_{2n}$ and $ M(q)_{2n-1}$. In  formula (\ref{e:2n-1}), the case $n=1$
  comes from the term $\delta_{i0}$ in  (\ref{eq:demoi1}).
  \end{proof}

For later use we define $\cU_n^{(i)}\subset\cQ_n$ as the subset of quadruplets $q$ such that the 
$2i$ dominant coefficients of $M(q)$ vanish,
\be  \label{e:Uni}
\cU_n^{(i)} = M^{-1}\big(P_{2n-2i}\big) \cap \cQ_n
\ee
Obviously, $q\in \cQ_n$ is in
$
\cU_n^{(i)}$
 if and only if 
 $M(q)_{2n-2i+1} = \dots = M(q)_{2n}  =  0, 
 $ and in particular $q \in \cU_n^{(n)}$ iff $M(q) \in \RR$.
 Thus one has the embeddings
\[
 \cU_n\subset\cU_n^{(n)}\subset\dots\subset\cU_n^{(1)}\subset \cQ_n.
\]

  \subsection{Metrics}
  
  The continuity properties of the projection algorithm defined in the next section will be analyzed with convenient
   norms which are defined below.
   
   For any real polynomial $p$,  we consider   its  weighted $L^1$ norm 
  \[
  \|p\| := \int_0^1|p(x)|\frac{\dx}{\sqrt{w(x)}}.
  \]
  For quadruplets $q=(a,b,c,d)\in\cQ_n$, we define a specific  norm $\nrm\cdot\nrm$  
  \begin{equation} \label{eq:def_nrmQ}
  \nrm q\nrm^2  := \|M(q)\|\ =\ \int_0^1a^2w^{-\frac{1}{2}} + \int_0^1b^2w^{\frac{1}{2}} + \int_0^1c^2w^{-\frac{1}{2}} + \int_0^1d^2w^{\frac{1}{2}}\,.
  \end{equation}

  The orthogonality of Chebychev polynomials yields the Plancherel-like equality 
  \begin{equation} \label{eq:Plancherel}
   \nrm q\nrm^2\ =\ \pi\,(|a_0|^2+|c_0|^2)+\frac{\pi}{2}\,\sum_{i=1}^n(|a_i|^2+|b_i|^2+|c_i|^2+|d_i|^2).
  \end{equation}
   Since  $M$ is a morphism and $M(e) =1$ for $e\in\cU_\infty$, one has 
  \be \label{unitprod}
  \nrm eq\nrm\ =\ \nrm q\nrm\quad\text{for any}\quad e\in\cU_\infty, \ q \in \cQ_\infty.
  \ee
  This last property is  very useful when dealing with the decomposition formulas of Theorem~\ref{th:decomp}.

  \section{The projection algorithm}\label{s:projection}
  
 In order to motivate our projection algorithm we consider the problem of computing 
 a polynomial approximation with two bounds 
 to some given function $f$
 assuming that,  as a preliminary step, we are able to construct two polynomials with one bound,
 $p_0 \in P^+_{2n}$ and $p_1 \in 1 - P^+_{2n}$, which both approximate $f$ in some sense,
 $$
 p_0 = a^2 \,+\, b^2w \approx f 
 \qquad\text{ and }\qquad 
 p_1 = 1 - c^2 \, - \, d^2 w \approx f.
 $$
 By construction, the polynomial $p_0$ is  non negative and the polynomial $p_1$ is less than 1.
 The point is that this preliminary step  is doable:  
  for example we refer to \cite{campos_2019_algorithms,despres_2018_iterative} where effective algorithms are proposed
  to compute polynomial approximations with one bound.
  The method \cite{campos_2019_algorithms} is restricted to monovariate polynomials, while \cite{despres_2018_iterative} is more general and adresses multivariate polynomials.
  In the numerical section we shall use a third different method described in the appendix. In all cases, one ends up with a quadruplet $q = (a,b,c,d) \in \cQ_n$
  such that
  \[
  M(q) = a^2 + b^2 w + c^2 + d^2w = p_0 + 1 - p_1 \approx 1.
  \]
 In particular, the quadruplet $q$ may not be in $\cU_n$, so that neither $p_0$ or $p_1$ are in $U_{2n}$.
 The numerical illustrations at the end show it is indeed the case.
 Our objective is then to construct an algorithm which projects 
 $q = (a,b,c,d)$ into $\tilde q = (\ta,\tb,\tc,\td)\in\cU_n$ and thus
 provides a polynomial approximation 
 \[
 \tilde{p} := \ta^2+\tb ^2 w = 1 - \tc^2-\td^2 w  \approx f
 \quad 
 \text{ with two bounds, } \quad
 \tilde{p} \in U_{2n}.
 \]
 To do so we will use the iterative decomposition technique developped in the theoretical proof of 
 \cite{despres_2016_polynomials_erratum,despres_2016_polynomials} with an additional
 correction step.

\subsection{Definition of the projection}

The design principle of the algorithm is to follow the iterative factorization structure 
developed in the proof of Theorem \ref{th:decomp}.
Since this procedure is  applied to a quadruplet that is {\em not} in the set $\cU_n$,
the key issue is to design a correction step that effectively allows to perform each iterative
factorization.
Thus our construction involves two functions that will be properly described below, see 
Definitions~\ref{def:chin}, \ref{def:chi1} and \ref{def:factor}.
  
\begin{itemize}
  \item \emph{The new correction function} $\chi_n: \cQ_{n} \rightarrow\cU_n^{(1)}$ is a projection 
  onto $\cU_n^{(1)}$, see \eqref{e:Uni}. From $q\in\cQ_n$
   it creates $\hq = \chi_n(q)$ by modifying only the two dominant coefficients of the four polynomials constituting $q$, in order for the two dominant coefficients of $M(\hq)$ to vanish. 
   
  \smallskip
  
  \item \emph{The factorization function} $\phi_n: \cU_n^{(1)}\rightarrow\cU_1$, which from a corrected quadruplet $\hq$ explicitly builds an element $e = \phi_n(\hq)\in\cU_1$ such that $e\hq\in\cQ_{n-1}$. It relies on  a technical 
   adaptation
  of  the proof of Theorem \ref{th:decomp}.
\end{itemize}

The structure of the algorithm is then as follows. 

\begin{defi}\label{def:proj}
The projection onto $\cU_n$ is defined by the factorized formula
\be \label{eq:proji1}
\Pi_n:
\left\{
\begin{array}{rcl}
 \cQ_n&\longrightarrow&\cU_n\\[.5em]
 q&\longmapsto&\ds \overline{e_1}\,\overline{e_2}\,\dots\,\overline{e_n}\,r_0
\end{array}
\right.
\ee
where each factor is computed iteratively,
setting $q_n := q$ and
for $i = 0, \dots, n-1$,
\be \label{eq:projloop}
\left\{
\begin{array}{rccl}
\hq_{n-i}&:= &\chi_{n-i}(q_{n-i}) &\in \, \cU_n^{(1)}\, , \\
e_{i+1}& :=& \phi_{n-i}(\hq_{n-i})&\in\, \cU_1\,,\\
q_{n-(i+1)}& :=& e_{i+1}\hq_{n-i}&\in \, \cQ_{n-(i+1)} . 
\end{array}
\right.
\ee
Here, $\chi_{n-i}$ is the correction function 
defined in Def.~\ref{def:chin} and \ref{def:chi1},
$\phi_{n-i}$ is the factorization function defined in Def.~\ref{def:factor} 
and $e_{i+1}\hq_{n-i}$ is a quaternion product.  
Finally the last term $r_0 \in \cU_0$
in (\ref{eq:proji1}) is defined as 
\[
 r_0\ :=\ \left\{\begin{array}{ll}\ds q_0 /M(q_0)^{1/2}&\text{ if }\ q_0\, \neq\, 0\,,\\ 
                  (1,0,0,0)&\text{ otherwise}\,.
                \end{array}
\right.
\]
\end{defi}

\subsection{The correction function $\chi_n: \cQ_{n}\rightarrow\cU_n^{(1)}$ for  $n\geq 2$}

Let $q = (a,b,c,d)\in \cQ_n$ and let us define $\chi_n(q) :=\hq =(\ha,\hb,\hc,\hd)\in \cU_n^{(1)}$. The polynomials $\ha$, $\hb$, $\hc$ and $\hd$ are defined by changing only the dominant coefficients 
of $(a,b,c,d)$ in the Chebychev basis~\eqref{e:chebas}. This is performed as follows.
 
  The low order coefficients remain unchanged, namely 
 \[
 \ha_i = a_i\,,\quad \hb_i = b_i\,,\quad \hc_i = c_i\,,\quad \ha_i = a_i\quad \text{ for all } i\leq n-2
\]
In order for $\hq$ to be an element of $\cU_n^{(1)}$, 
the remaining high order coefficients must satisfy the algebraic relations \eqref{e:2n}-\eqref{e:2n-1}
\be
 \left\{
  \begin{array}{l}
  \ha_n^2 \,-\, \hb_n^2 \,+\, \hc_n^2 \,-\, \hd_n^2\ =\ 0\,,\\[.5em]
  \ha_n\,\ha_{n-1}\, -\, \hb_n\,\hb_{n-1}\, +\, \hc_n\,\hc_{n-1}\, -\, \hd_n\,\hd_{n-1}\ =\ 0\,.
\end{array}
\right.
\label{e:corrected}
\ee
Since we desire $\chi_n(q)$ to be as close as possible to $q$, we decide to project
\be\label{e:defX}
 X = \left(a_{n}, a_{n-1}, b_{n}, b_{n-1}, c_{n}, c_{n-1}, d_{n}, d_{n-1}\right)^t
\ee
onto the algebraic manifold $\mathcal{V}\subset \mathbb R^8$ defined by \eqref{e:corrected}. The problem is thus reduced to building a projection $\chi: \mathbb{R}^8\rightarrow\mathcal{V}$.

The mathematical issue is that the Euclidean projection on this non-convex set cannot be properly defined. Indeed  if one denotes by $\|\cdot\|$ the euclidean norm in $\mathbb R^8$ the following quadratically constrained quadratic program
\be
 \inf_{Y\in\mathcal{V}}\ \frac12\|X-Y\|^2\,,
 \label{e:primal}
\ee
may have multiple solutions.   Via a dual convex nonlinear program, 
 we are nevertheless able to explicitly compute at least one solution, which reveals sufficient for
 our algorithmic purposes. Once we are provided with a suitable candidate written as
 \[
  \chi(X) = \left(\ha_{n}, \ha_{n-1}, \hb_{n}, \hb_{n-1}, \hc_{n}, \hc_{n-1}, \hd_{n}, \hd_{n-1}\right)^t\,,
 \]
  we may gather the coefficients to determine $\chi_n(q)$. This will properly stated in Definition~\ref{def:chin}.
 
\subsubsection{A dual convex program}

The set of constraints of the optimization problem \eqref{e:corrected} is written as 
\[
\mathcal{V}\ =\ \{Y\in\RR^8 \ \text{ such that } \ {Y}^tA{Y}\, =\, {Y}^tB{Y} \, =\, 0\}
\]
with symmetric block diagonal matrices 
\[
A = \mathrm{diag}(S, -S, S, -S)\, \in \cM_8(\RR),\quad B = (T,-T,T,-T)\, \in \cM_8(\RR)
\]
where
\[
 S = \left(\begin{matrix}
            1&0\\
            0&0
           \end{matrix}
\right)\,,\quad T = \left(\begin{matrix}
            0&1\\
            1&0
           \end{matrix}
\right)\,.
\]
The set $\mathcal{V}$ is also called the correction manifold in the following.
The  Lagrangian associated to \eqref{e:primal} is 
\[
  L(Y,\lambda,\mu)\ =\ \frac12\left(\|X-Y\|^2 + \lambda\,{Y}^tA{Y} + \mu\,{Y}^tB{Y}\right).
\]
The triplets $(Y, \lambda, \mu)$ satisfying the  first order optimality condition $\nabla L\,=\,0$ are those satisfying $Y\in\mathcal{V}$ and
\[
  M_{\lambda,\mu}Y\ =\ X
 \]
 with
 \be\label{e:defM}
 M_{\lambda,\mu}\ =\ I + \lambda\,A + \mu\,B\,.
 \ee
The conditions of invertibility of $ M_{\lambda,\mu}$ reduce to the invertibility of
$
I \pm (\lambda S + \mu T)
$. 
The four eigenvalues of the symmetric matrix
$M_{\lambda,\mu}=M_{\lambda,\mu}^t\in \cM_8(\mathbb R)$ are
$
1 \pm  (\lambda \pm \sqrt{|\lambda|^2 + 4|\mu|^2})/2$ and $1 \pm  (\lambda \mp \sqrt{|\lambda|^2 + 4|\mu|^2})/2$. 
It is natural to define the open and convex set 
\[
 \mathcal{D}=\{(\lambda,\mu)\in\RR^2\ \text{such that}\  M_{\lambda,\mu}>0\} = \{ |\lambda|+\mu^2 < 1 \}.
\]
This set is bounded 
with boundary $\partial\mathcal{D}=\{ |\lambda|+\mu^2=1 \}$.
Moreover, on $\mathcal{D}$ it holds 
$$
I \pm (\lambda S + \mu T) \geq 0,
\quad \text{ hence } \quad 
\left\|
\lambda S + \mu T
\right\|\leq 1
$$
in the matrix 2-norm over $\RR^2$, which results in a uniform bound
\[
\left\|M_{\lambda,\mu}  \right\|\leq 2, \qquad (\lambda,\mu)\in \mathcal{D}
\]
in the matrix 2-norm over $\RR^8$.
  Let us now consider the dual optimization problem 
 \be
  (\lambda^*(X),\mu^*(X))\in\mathrm{arg}\inf_{(\lambda,\mu)\in\mathcal{D}}G_X(\lambda,\mu)
  \label{e:dual}
 \ee
 with
 \be\label{def_G}
  G_X(\lambda,\mu) = X^t\,M_{\lambda,\mu}^{-1}\,X\,.
 \ee
 The function $G_X$ enjoys the following nice property.
 
 \begin{lemma} \label{lem:crit}
 Assume $G_X$ has a critical point $(\lambda^*,\mu^*) \in \mathcal D$, in the sense that
 $\nabla G_X(\lambda^*,\mu^*)=0$. 
 Then 
 $ Y^*= M_{\lambda^*,\mu^*}^{-1}X$ is in the correction manifold $\mathcal V$. 
 \end{lemma}
 
 \begin{proof}
 One has the differential  formula
  $dM^{-1}=-M^{-1}dM M^{-1}$ which holds for matrices $M>0$. 
So an explicit calculation shows that 
$$
\partial_\lambda G_X(\lambda^*, \mu^*) = -
X^t
  M_{\lambda^*,\mu^*}^{-1} AM_{\lambda^*,\mu^*}^{-1}X = - {Y^*}^tA{Y^*} = 0
  .
$$ 
Similarly 
 $\partial_\mu G_X(\lambda^*, \mu^*) = -{Y^*}^tB{Y^*} = 0$,
 hence $Y^*\in\cV$. In particular, $\cV\neq \emptyset$.
 \end{proof}

 The following  result shows that, generically,   $(\lambda^*, \mu^*)$   exists and is 
 a global minimum of the functional $G_X$.
 
 \bl \label{l:propG}
 For any  $X\in\RR^8$, the function $G_X : \mathcal{D}\rightarrow\RR_+$ is convex and $C^1$. Moreover there is a dense open subset $\mathcal{S}\subset \RR^8$ such that whenever $X\in\mathcal{S}$, the function $G_X$  tends to $+\infty$ on the boundary of $\cD$ (namely, it is coercive).

 \el
 \begin{proof}
  The convexity stems from the non-negativity of $M_{\lambda,\mu}$ since
  for $\alpha,\beta\in\RR^2$
  \[
  \binom{\alpha}{\beta}^t\mathrm{Hess}\,G_X(\lambda, \mu)\binom{\alpha}{\beta} = X^t
  M_{\lambda,\mu}^{-1}(\alpha A+\beta B)
  M_{\lambda,\mu}^{-1}(\alpha A+\beta B)M_{\lambda,\mu}^{-1}X \geq 0.
 \]
 By explicitly inverting $M_{\lambda,\mu}$ and using the notation \eqref{e:defX}, one has that
 $$
 \begin{array}{lllll}
 G_X(\lambda,\mu) &= \displaystyle \frac{(a_n-\mu a_{n-1})^2}{1+\lambda-\mu^2} &+\ a_{n-1}^2 
                                  & +\  \displaystyle  \frac{(b_n+\mu b_{n-1})^2}{1-\lambda-\mu^2} &+\  b_{n-1}^2 \\
                                & + \   \displaystyle  \frac{(c_n-\mu c_{n-1})^2}{1+\lambda-\mu^2} &+\  c_{n-1}^2 
                                  & +\   \displaystyle \frac{(d_n+\mu d_{n-1})^2}{1-\lambda-\mu^2} &+\  d_{n-1}^2 .
 \end{array}
 $$
 This shows that $G_X$ is $C^1$ on $\mathcal{D}$ and goes to $+\infty$ on 
 $\partial\mathcal{D}=\{(\lambda,\mu)\ \text{s.t.}\ |\lambda|+\mu^2 = 1\}$ as soon as 
  the terms between parenthesis do not vanish (uniformly in $\mu$). It is the case for 
 \[
X\in   \mathcal{S}\ =\ \left\{
 {a_n}{c_{n-1}}\neq 
  {a_{n-1}}{c_{n}}\ \text{ and }\ {b_n}{d_{n-1}}\neq {b_{n-1}}{d_{n}}\right\}\subset \mathbb R^8.
 \]
 The set $\cS$ is an open and dense subset of $\RR^8$. 
 \end{proof}

At this point, for any $X$ in the dense set $\cS$ of Lemma~\ref{l:propG}, we know that the dual optimization problem admits at least one solution $(\lambda^*, \mu^*)\in\mathcal{D}$ that is a critical point of $G_X$. We can then define 
\be\label{e:chi}
\begin{array}{ll}
\chi(X): & \ \mathcal S \longrightarrow  \mathbb \cV, \\
 & X\longmapsto \chi(X) = M_{\lambda^*(X),\mu^*(X)}^{-1}X
\end{array}
\ee
where $(\lambda^*(X), \mu^*(X))$ is a global minimizer of $G_X$ obtained by a given convex optimization method. Of course, the definition of $\chi$ may vary since there are possibly several global minima 
(the precise implementation is detailed in Section~\ref{s:implement}). 
Also by perturbation around $\mathcal S$, the function $\chi$ can  defined 
$$
\chi(X): \quad  \ \mathbb R^8 \longrightarrow  \cV
$$
with the same restrictions concerning the choice of the minimizer which is non unique as well
and the choice of the perturbation. Regardless of these choices, we may now state the complete definition of $\chi_n$ when $n\geq2$.

\begin{defi}\label{def:chin}
The function $\chi_n: \cQ_{n}\rightarrow\cU_n^{(1)}$ takes $q =(a(x),\, b(x),\, c(x),\, d(x))$ 
as argument and returns
\[
 \chi_n(q) = \hq = (\ha(x),\, \hb(x),\, \hc(x),\, \hd(x))
\]
with
\[
 \ha_i = a_i\,,\quad \hb_i = b_i\,,\quad \hc_i = c_i\,,\quad \hd_i = d_i\quad \text{ for all }\  i\leq n-2
\]
and
\[
  \left(\ha_{n}, \ha_{n-1}, \hb_{n}, \hb_{n-1}, \hc_{n}, \hc_{n-1}, \hd_{n}, \hd_{n-1}\right) = \chi\left(a_{n}, a_{n-1}, b_{n}, b_{n-1}, c_{n}, c_{n-1}, d_{n}, d_{n-1}\right)
\]
where the projection $\chi$ is defined in \eqref{e:dual}-\eqref{e:chi}. 
\end{defi}

\subsubsection{Properties of the non convex optimization problem}

\bpr\label{p:estim_chi}
The function $\chi$ has values in the correction manifold $\cV$, and 
\begin{itemize}
\item[(i)] it is nonincreasing in the euclidean norm of $\RR^8$, namely 
\[
\|\chi(X)\|\,\leq\,\|X\| ;
\]
\item[(ii)] it satisfies the estimate
\[
 \| X-\chi(X)\|\ \leq\ 2^{\frac 34} \|X\|^{1/2}(\left|{X}^tA{X}\right| + \left|{X}^tB{X}\right|)^{1/4}
\]
where the right hand side vanishes for $X\in \cV$;
\item[(iii)] it is idempotent, i.e. $\chi\circ\chi=\chi$.
\end{itemize}
These estimates are uniform with respect to the choice of the minimizer in \eqref{e:dual}.
\epr
\begin{proof} Let $X\in\cS$ and 
$Y^* = \chi(X) = M_{\lambda^*,\mu^*}^{-1}X$ as defined in Lemma \ref{lem:crit}. We know that 
$Y^* \in \cV$.
\begin{itemize}
\item[(i)] One has
$$
X^tY^* = {Y^*}^tM_{\lambda^*,\mu^*}Y^* =
{Y^*}^t (I+\lambda^*A+ \mu^* B)Y^* = \|Y^*\|^2
$$
 which yields the first estimate $\|Y^* \| \leq \|X\|$.

\item[(ii)] A Taylor formula with integral remainder expansion yields
 $$
  G_X(\lambda^*, \mu^*) 
  =\ds G_X(0,0) - (\lambda^*\,{X}^tA{X} \,+\, \mu^*\,{X}^tB{X})
  $$
  $$
 \ds+ 2\int_0^1X^tM_{s\lambda^*,s\mu^*}^{-1}(\lambda^*A+\mu^*B)M_{s\lambda^*,s\mu^*}^{-1}(\lambda^*A+\mu^*B)M_{s\lambda^*,s\mu^*}^{-1}X(1-s)\,\dD s.
 $$
Since $G_X(\lambda^*, \mu^*)\leq G_X(0,0)$ and the matrices commute
$$
M_{s\lambda^*,s\mu^*}^{-1}(\lambda^*A+\mu^*B)
=(\lambda^*A+\mu^*B)M_{s\lambda^*,s\mu^*}^{-1},
$$
one has the inequality
$$
2\int_0^1 (Z^*) ^tM_{s\lambda^*,s\mu^*}^{-3} Z^*(1-s) \dD s
\ \leq\ (\lambda^*\,{X}^tA{X} \,+\, \mu^*\,{X}^tB{X})
$$
where $ 
Z^*=(\lambda^*A+\mu^*B)
X
$.
It yields 
$$
\|Z^*\|^2=
2\int_0^1 (M_{s\lambda^*,s\mu^*}^{-3/2} Z^*) ^t M_{s\lambda^*,s\mu^*}^{3}M_{s\lambda^*,s\mu^*}^{-3/2} Z^*(1-s) \dD s
$$
$$
\leq 2 
\int_0^1 \|M_{s\lambda^*,s\mu^*}\| ^{3} \left\| M_{s\lambda^*,s\mu^*}^{-3/2} Z^*\right\|^2 (1-s) \dD s
$$
$$
\leq 2 
\int_0^1 2 ^{3} \left\| M_{s\lambda^*,s\mu^*}^{-3/2} Z^*\right\|^2 (1-s) \dD s
$$
$$
\leq 2^4 \int_0^1 (Z^*) ^tM_{s\lambda^*,s\mu^*}^{-3} Z^*(1-s) \dD s
$$
$$
 \leq 2^3 (\lambda^*\,{X}^tA{X} + \mu^*\,{X}^tB{X}).
 $$
 Using $| \lambda^*|+(\mu^*)^2<1$, one gets the technical bound
 $$
 \|Z^*\| \leq 2^\frac32 (\left|{X}^tA{X}\right| + \left|{X}^tB{X}\right|)^{1/2}.
 $$
By definition of $Y^*$ one has $X-Y^*=(\lambda^*A+\mu^*B) Y^*$. So
\[
\begin{array}{rcl}
 \|X-Y^*\|^2 &=& {Y^*}^t(\lambda^*A+\mu^*B)(X-{Y^*})\\[1em]
 &=& {Y^*}^t(\lambda^*A+\mu^*B)X\\[1em]
 &=& {Y^*}^tZ^*\leq  \|Y^*\|\,\|Z^*\|. 
\end{array}
\]
So $ \|X-Y^*\|\leq \|Y^*\| ^\frac12 \|Z^*\|^\frac12$.
One concludes with (i) and the previous technical bound.

\item[(iii)] The estimate in (ii) yields $\|\chi\circ\chi(X)-\chi(X)\| = 0$ since $\chi(X)\in\cV$.
\end{itemize}
The proof is ended.
\end{proof}

\begin{corollary} \label{cor:prop_chin}
Let $n\geq2$. The correction function $\chi_n:\cQ_n\rightarrow \cU_n^{(1)}$ satisfies
$$
\begin{array}{llll}
\emph{(i)} &   \nrm \chi_n(q)\nrm \leq \nrm q\nrm,  & q\in \cQ_n, \\[.5em]
\emph{(ii)} &   \nrm q- \chi_n(q)\nrm \leq\ C\, \nrm q\nrm^{1/2}(\left|M(q)_{2n}\right| \,+\, \left|M(q)_{2n-1}\right|)^{1/4} , &  q\in \cQ_n,\\[.5em]
\emph{(iii)} & \chi_n \circ \chi_n = \chi_n.
\end{array}
$$
for some constant $C>1$.
\end{corollary}
\begin{proof}
 These properties follow from Proposition~\ref{p:estim_chi},
 observing that the non zero coefficients of $q- \chi_n(q)$ 
 coincide with those of $X-\chi(X)$: using \eqref{eq:Plancherel} this gives
 $$
 \nrm q\nrm^2 - \nrm \chi_n(q)\nrm^2= \frac\pi 2 \big(\|X\|^2 - \|\chi(X)\|^2\big) \ge 0
 $$
 and for estimate (ii) we use
 $\left|M(q)_{2n}\right| + \left|M(q)_{2n-1}\right|=
\frac 12 \big( \left|{X}^tA{X}\right| + \left|{X}^tB{X}\right|\big)$.
 \end{proof}

\subsection{The correction function $\chi_1: \cQ_{1}\rightarrow\cU_1^{(1)}$} 
For $n = 1$, the correction function $\chi_n$ needs a specific definition.
Indeed, in order for $\hq = \chi_1(q)$ to be in $\cU_1^{(1)}$, 
the following relations must hold
\be
 \left\{
  \begin{array}{l}
  \ha_1^2 \,-\, \hb_1^2 \,+\, \hc_1^2 \,-\, \hd_1^2\ =\ 0,\\[.5em]
  \ha_1\,\ha_{0}\, +\, \hc_1\,\hc_{0}\ =\ 0,
\end{array}
\right.
\label{e:corrected_negal1}
\ee
and they slightly differ from the previous ones (\ref{e:corrected}). 
However the method and results are essentially the same. Specifically, (\ref{e:corrected_negal1}) 
define a slightly different  set of constraints 
\[
 \tcV\ =\ \{\tilde{Y} = (\ha_1, \ha_0, \hb_1,\hc_1,\hc_0,\hd_1)\in\mathbb{R}^6
 \ \text{ such that } \ 
 \tilde{Y}^t\tilde{A}\tilde{Y}\ =\ \tilde{Y}^t\tilde{B}\tilde{Y} =\ 0\}
\]
with symmetric block diagonal matrices 
\[
\tilde{A} = \mathrm{diag}(S, -1, S, -1)\, \in \cM_6(\RR),\quad \tilde{B} = (T,0,T,0) \, \in \cM_6(\RR).
\]
This leads to the dual optimization problem
 \be
  (\lambda^*(\tilde{X}),\mu^*(\tilde{X}))\in\mathrm{arg}\inf_{(\lambda,\mu)\in\widetilde{\mathcal{D}}}\tilde{G}_{\tilde{X}}(\lambda,\mu)\qquad\text{with}\qquad \tilde{G}_{\tilde{X}}(\lambda,\mu) = \tilde{X}^t\,\tilde{M}_{\lambda,\mu}^{-1}\,\tilde{X}\,
  \label{e:dual_negual1}
 \ee
with a matrix $\tilde{M}_{\lambda,\mu} = I + \lambda\tilde{A} + \mu\tilde{B}$ and a 
bounded convex domain now defined as $\widetilde{\mathcal{D}} = \{(\lambda, \mu ) \in \RR^2 :  \mu^2 -1\leq\lambda\leq 1\}$. Thus we define 
\be
\begin{array}{ll}
\tilde{\chi}: & \  \mathbb R^6 \longrightarrow  \mathbb \tcV, \\
 & \tilde{X}\longmapsto \tilde{\chi}(\tilde{X}) = \tilde{M}_{\lambda^*(\tilde{X}),\mu^*(\tilde{X})}^{-1}\tilde{X}
\end{array}
\label{e:chi_negal1}
\ee
where $(\lambda^*(\tilde{X}),\mu^*(\tilde{X}))$ is the global minima of the convex and coercive nonlinear program \eqref{e:dual_negual1} obtained by a given optimization method.

\begin{defi}\label{def:chi1}
The function $\chi_1: \cQ_{1}\rightarrow\cU_1^{(1)}$ takes $q =(a_1T_1(x) + a_0,\, b_1U_1,\, c_1T_1(x) + c_0,\, d_1U_1)$ as argument and returns
\[
 \chi_1(q) = \hq = (\ha_1\, T_1(x) + \ha_0, \hb_1\, U_1,\hc_1\, T_1(x) + \hc_0, \hd_1\, U_1 )
\]
with $(\ha_1, \ha_0, \hb_1,\hc_1,\hc_0,\hd_1) = \tilde{\chi}(a_1, a_0, b_1,c_1,c_0,d_1)$ and 
$\tilde{\chi}$ defined by (\ref{e:dual_negual1}-\ref{e:chi_negal1}). 
\end{defi}

The function $\chi_1$ has the same properties as $\chi_n$ for $n\geq 2$.
In particular the results of Corollary~\ref{cor:prop_chin} can be established
also for $n=1$. We state this as a proposition for later reference.

\begin{proposition}\label{prop:prop_chi_all}
Let $n\geq1$. The correction function $\chi_n:\cQ_n\rightarrow \cU_n^{(1)}$ satisfies
$$
\begin{array}{llll}
\emph{(i)} &   \nrm \chi_n(q)\nrm \leq \nrm q\nrm,  & q\in \cQ_n, \\[.5em]
\emph{(ii)} &   \nrm q- \chi_n(q)\nrm \leq\ C\, \nrm q\nrm^{1/2}(\left|M(q)_{2n}\right| \,+\, \left|M(q)_{2n-1}\right|)^{1/4} , &  q\in \cQ_n,\\[.5em]
\emph{(iii)} & \chi_n \circ \chi_n = \chi_n.
\end{array}
$$
for some constant $C>1$.
\end{proposition}

\subsection{The factorization function $\phi_n: \cU_n^{(1)}\rightarrow\cU_1$ for $n\geq 1$}

\begin{defi}\label{def:factor}
The factorization 
function $\phi_n: \cU_n^{(1)}\rightarrow\cU_1$ takes $\hq=(\ha,\hb,\hc,\hd)$ as argument. 
If $\ha_n^2 + \hc_n^2 = 0$, it returns $\phi_n(\hq) = (1,0,0,0)$. Otherwise it is defined as follows.
\begin{description}
  
\item[Case $n\geq2$]  then 
$
\phi_n(\hq) = K\, (\alpha_1 T_1 + \alpha_0,\  \beta_1\,U_1,\ \gamma_1\,T_1 + \gamma_0,\  \delta_1\,U_1)$ 
where 
\be
\alpha_1\,=\,\ha_n\,,\qquad \beta_1\, =\, -\,\hb_n\,,\qquad \gamma_1\, =\, -\,\hc_n\,,\qquad \delta_1\, =\, -\,\hd_n\,,
\label{e:dom_coeff_factor}
\ee
\be\label{e:alpha0}
 \alpha_0= \frac{\ha_{n-1}}{2} \ -\  \frac{\hb_{n}\,\hb_{n-1} + \hd_{n}\,\hd_{n-1}}{2(\ha_n^2 \,+\, \hc_n^2)}\,\ha_n \ +\  \frac{\hb_{n}\,\hd_{n-1} \, -\, \hd_{n}\,\hb_{n-1} }{2(\ha_n^2 \,+\, \hc_n^2)}\,\hc_n\,,
 \ee
 \be\label{e:gamma0}
 \gamma_0= \ds-\,\frac{\hc_{n-1}}{2} \ +\  \frac{\hb_{n}\,\hb_{n-1} + \hd_{n}\,\hd_{n-1}}{2(\ha_n^2 \,+\, \hc_n^2)}\,\hc_n \ +\  \frac{\hb_{n}\,\hd_{n-1} \, -\, \hd_{n}\,\hb_{n-1} }{2(\ha_n^2 \,+\, \hc_n^2)}\,\ha_n\,,
\ee
and
$
K= \left(\alpha_0^2 + \gamma_0^2 + \frac 12(\alpha_1^2 + \beta_1^2 + \gamma_1^2 + \delta_1^2)\right)^{-1/2}$
which is correctly defined since $\alpha_1^2 + \gamma_1^2>0$.
\item[Case $n=1$] then 
   $\phi_1(\hq) =K  {\hq}$ with  
   $$
   K=M(\hq)^{-1/2}= \left(\ha_{0}^2 + \hc_{0}^2 + \frac 12(\ha_{1}^2 + \hb_{1}^2 + \hc_{1}^2 + \hd_{1}^2)\right)^{-1/2}.
   $$

\end{description}
\end{defi}

\begin{rema}\label{rem:already_factorized}
 If  $\ha_n^2 + \hc_n^2 = 0$ then $\hq\in\cU_n^{(1)}$. So by \eqref{e:corrected} (or \eqref{e:corrected_negal1} if $n=1$), one has also $\hb_n^2 + \hd_n^2 = 0$ and thus $\hq\in\cQ_{n-1}$. This explains why these cases are distinguished in the definition.  
\end{rema}

\begin{proposition}\label{prop_factor}
 For all $\hq\in\cU_n^{(1)}$, one has $\phi_n(\hq)\in\cU_1$ and $\phi_n(\hq)\hq\in\cQ_{n-1}$.
\end{proposition}
\begin{proof}
If $n=1$, then since $\hq\in\cU_1^{(1)}$, one has  $M(\hq)_2 = M(\hq)_1 = 0$, so clearly $\phi_1(\hq)\in\cU_1$ and $\phi_1(\hq)\hq\in\cQ_{0}$.
 Consider  the product formulas \eqref{e:foursquare_coeff}.
 Regardless of the values of $\alpha_0$ and $\gamma_0$,  the product $(A, B, C, D) =(\alpha, \beta, \gamma, \delta)\,\hq$ is such that $B_{n+1} = C_{n+1} = D_{n+1} = 0$; thanks to \eqref{e:corrected} one also has $A_{n+1} = 0$. The next coefficients of $(A, B, C, D)$  are 
 \[
\begin{array}{rcl}
  2\,A_{n} &=&\ds \Big{(}\ha_{n}\,\ha_{n-1} - \hb_{n}\,\hb_{n-1} + \hc_{n}\,\hc_{n-1} - \hd_{n}\,\hd_{n-1} \Big{)}+ 2\ (\alpha_0\,\ha_n - \gamma_0\,\hc_n)\,,\\[.5em]
  2\,B_{n} &=&\ds \Big{(}- \,\hb_{n}\,\ha_{n-1} + \ha_{n}\,\hb_{n-1} + \hd_{n}\,\hc_{n-1} - \hc_{n}\,\hd_{n-1}\Big{)}+ 2\ (\alpha_0\,\hb_{n} + \gamma_0\,\hd_{n})\,,\\[.5em]
  2\,C_{n} &=&\ds \Big{(}-\,\hc_{n}\,\ha_{n-1} + \hd_{n}\,\hb_{n-1} + \ha_{n}\,\hc_{n-1} - \hb_{n}\,\hd_{n-1}\Big{)}+ 2\ (\gamma_0\,\ha_{n} + \alpha_0\,\hc_{n})\,,\\[.5em]
  2\,D_{n} &=&\ds \Big{(}-\,\hd_{n}\,\ha_{n-1} + \hc_{n}\,\hb_{n-1} - \hb_{n}\,\hc_{n-1} + \ha_{n}\,\hd_{n-1} \Big{)}+ 2\ (-\, \gamma_0\,\hb_{n} + \alpha_0\,\hd_{n})\,.
\end{array}
\]
Thanks to the choice of $\alpha_0$ and $\gamma_0$ in \eqref{e:alpha0} and \eqref{e:gamma0}, all these coefficients vanish too. To simplify the computation of $B_n$ and $D_n$, notice that since $\ha_n^2 + \hc_n^2 = \hb_n^2 + \hd_n^2$, the coefficients $\alpha_0$ and $\gamma_0$ rewrite
\[
 \alpha_0= \frac{\ha_{n-1}}{2} \ +\  \frac{\hc_{n}\,\hd_{n-1} - \ha_{n}\,\hb_{n-1}}{2(\hb_n^2 \,+\, \hd_n^2)}\,\hb_n \ -\  \frac{\ha_{n}\,\hd_{n-1} \, +\, \hc_{n}\,\hb_{n-1} }{2(\hb_n^2 \,+\, \hd_n^2)}\,\hd_n\,,
\]
and 
\[
 \gamma_0= -\frac{\hc_{n-1}}{2}\ +\  \frac{\hc_{n}\,\hd_{n-1} \, -\, \ha_{n}\,\hb_{n-1} }{2(\hb_n^2 \,+\, \hd_n^2)}\,\hd_n\ +\  \frac{\ha_{n}\,\hd_{n-1} + \hc_{n}\,\hb_{n-1}}{2(\hb_n^2 \,+\, \hd_n^2)}\,\hb_n \,.
\]
Thus $\phi_n(\hq)\hq\in\cQ_{n-1}$. Finally  $\phi_n(\hq)\in\cU_1$ since
\begin{multline*}
 M\left(K(\alpha, \beta, \gamma, \delta)\right)(x)\ =\ \frac{K^2}{2}\,(\ha_n^2 \,-\, \hb_n^2 \,+\, \hc_n^2 \,-\, \hd_n^2)\, T_2(x) \\
 + \frac{K^2}{2}\,(\ha_n\ha_{n-1} \,-\, \hb_n\hb_{n-1} \,+\, \hc_n\,\hc_{n-1} \,-\, \hd_n\,\hd_{n-1})\, T_1(x) \\
 + K^2\,\left(\alpha_0^2 + \gamma_0^2 + \frac 12 (\ha_n^2 \,+\, \hb_n^2 \,+\, \hc_n^2 \,+\, \hd_n^2)\right)\, T_0\,\ =\ 1\,.\\
\end{multline*}
This ends the proof.
\end{proof}

\begin{rema} \label{rem:3.3}
 The factorization built in the previous proof provides a constructive proof of  Theorem \ref{th:decomp}. Indeed if $M(q)=1$, one can check
 that  the correction step is not active in the projection algorithm (\ref{eq:projloop}), i.e., 
 $\hq_{n-i}= q_{n-i}$ for all $i$. 
One recovers the decomposition formulas of Theorem \ref{th:decomp}.
\end{rema}

\section{Error estimates}\label{s:analysis}

 With the material developed above, one can now use the projection
and consider
$\Pi_n(q)=(\tilde a, \tilde b, \tilde c, \tilde  d)\in \cU_n$ for $q=(a,b,c,d)\in \cQ_n $. But for practical purposes,
which ultimately is our concern, such a procedure would have little interest if the difference $q-\Pi_n(q)$ was large.
It is precisely the purpose of this section to analyze this difference.

Since the projection algorithm is very nonlinear, one can expect technical difficulties in proving sharp error estimates. In what follows, we explain how the various estimates and properties already obtained
combine to show some continuity properties of the projection $\Pi_n$.

In order to quantify the distance to  $\cU_n$, we define the difference 
\be\label{e:error_pol}
 \eps(q)=M(q)-1.
\ee
The main theoretical result of this work is as follows. 

\bt \label{t:error}
 Let $n \in \NN$ and $H>0$. For any quadruplet $q\in\cQ_n$ satisfying $\nrm q \nrm\leq H$, one has
\[
\nrm q - \Pi_n(q)\nrm\,\leq\,(n+1)\,C(H)\,
\max\left\{ \|\eps(q)\|,  \|\eps(q)\|^{2^{-(2n+1)}} \right\}.
\]
for some constant $C(H)>0$ depending only on $H$.
\et

It is instructive to reformulate Theorem~\ref{t:error} in terms of polynomials rather than in terms of quaternions. 
     
\bc [of Theorem  \ref{t:error}] \label{t:error2}
 Let $n \in \NN$, $H > 0$ and $q = (a,b,c,d) \in \cQ_n$ an arbitrary quadruplet satisfying $\nrm q \nrm\leq H$.
 Note $p_0 = a^2 + b^2 w$ and  $p_1 = 1 - c^2 - d^2w$ and consider
  $ (\tilde{a}, \tilde{b},\tilde{c}, \tilde{d}) = \Pi_n(q)$.  There exists a constant  $C(H)>0$ such that 
 the polynomial with two bounds 
 $$
 \tilde p := \tilde{a}^2 + \tilde{b}^2 w = 1 - \tilde{c}^2 - \tilde{d}^2 w \ \in U_{2n}
 $$
 satisfies
  \[
  \|p_0 - \tilde p\|\ \leq\ (n+1)\, C(H)\,\max\left\{ \|p_0 - p_{1}\|, \, \|p_0-p_{1}\|^{2^{-(2n+1)}} \right\}.
  \]
\ec

\begin{proof}
 Using the definition of the norms $\|\cdot\|$, $\nrm\cdot\nrm$ and two Cauchy-Schwarz inequalities, we write
 \[
 \begin{array}{rcl}
 \|p_0-\tilde p\| &=& \|(a+\tilde{a})(a-\tilde{a}) + (b+\tilde{b})(b-\tilde{b})w \|
 \\[.5em]
 &=& \int_0^1 |(a+\tilde{a})(a-\tilde{a})w^{-\frac12} + (b+\tilde{b})(b-\tilde{b}) w^{\frac12}|
 \\[.5em]
 &\leq& \| (a+\tilde{a})^2 + (b+\tilde{b})^2 w\|^{\frac12}    
        \| (a-\tilde{a})^2 + (b-\tilde{b})^2 w\|^{\frac12}
 \\[.5em]
 &\leq& \nrm q + \Pi_n(q)\nrm\,\nrm q-\Pi_n(q)\nrm
\\[.5em]
 &\leq& (H + \nrm \Pi_n(q)\nrm)\,\nrm q-\Pi_n(q)\nrm.
 \end{array}
 \]
 The result follows by 
 combining the estimate of Theorem~\ref{t:error} with the equality $ \eps(q) = M(q)-1 = p_0 - p_1$
 and the observation that
 $\nrm \Pi_n(q)\nrm = \|M(\Pi_n(q))\|^{1/2} = \| 1\|^{1/2} = \sqrt{\pi}$.
 
\end{proof}

To prove Theorem \ref{t:error} we begin by establishing a couple of elementary estimates.

\bpr\label{p:estim_chin}
 There is a constant $C>1$ such that for any integer $n\geq1$ and any $q\in\cQ_n$, 
 the nonlinear correction operator $\chi_n$  satisfies
  \be
  \nrm q - \chi_n(q)\nrm\ \leq\ C\,\nrm q\nrm^{1/2}\,\|\eps(q)\|^{1/4}
  \label{e:estim_chi2}
  \ee
  as well as
  \be
   \|\eps(\chi_n(q))\|\ \leq\ C\,(1 + \nrm q \nrm^{3/2})\,\|\eps(q)\|^{1/4}.
  \label{e:estim_chi3}
   \ee
\epr

  \begin{proof} 
   Let $q = (a,b,c,d)$ and $\hq = (\ha,\hb,\hc,\hd) = \chi_n(q)$. 
   The first estimate follows from Proposition~\ref{prop:prop_chi_all}, 
   and the observation that the $i$-th coefficients of $M(q)$ and $\eps(q) = M(q)-1$ in the $(T_n)$ Chebyshev 
   basis coincide for $i \ge 1$, thus
  \[
  |M(q)_i| = |\eps(q)_i| = \tfrac 2 \pi |\lla \eps(q),\,T_i\rra_T| 
  \le \tfrac 2 \pi \|\eps(q)\| \|T_i\|_{L^\infty(0,1)} = \tfrac 2 \pi \|\eps(q)\|.
  \]   
  For the second estimate we compute
  \[
  \begin{aligned}
  \|\eps(\hq)\| & = \|\ha^2 + w\hb^2 + \hc^2 + w\hd^2-1\|\\[0.5em]
  & = \|\eps(q) + (\ha+a)(\ha-a) + w(\hb+b)(\hb-b) + (\hc+c)(\hc-c) + w(\hd+d)(\hd-d)\|\\[0.5em]
  & \leq \|\eps(q)\| + \nrm q + \hq \nrm \, \nrm q - \hq \nrm \\[0.5em]
  & \leq C \big(\|\eps(q)\|^{3/4} + \nrm q + \hq \nrm \,\nrm q\nrm^{1/2} \big)  \|\eps(q)\|^{1/4}
  \end{aligned}
  \]
  where the first inequality is obtained like in the proof of Corollary~\ref{t:error2},
  and the second one is \eqref{e:estim_chi2}.
  Finally estimate \eqref{e:estim_chi3} is  obtained by using 
  $\nrm \hq \nrm \le \nrm q \nrm$ from Proposition~\ref{prop:prop_chi_all}, and the bound
  $
  \|\eps(q)\| \le \|1\|+\|M(q)\| = \sqrt \pi + \nrm q \nrm^2
  $.
  \end{proof}
  With the estimates of Proposition~\ref{p:estim_chin} in hand, we can now prove Theorem~\ref{t:error}.
  
  \begin{proof}[Proof of Theorem~\ref{t:error}]
  For $q\in\cQ_n$ we write $\Pi_n(q) = \overline{e_1}\,\overline{e_2}\,\dots\,\overline{e_n}\,r_0$,
  according to Definition~\ref{def:proj}. 
  Using that $\nrm eq\nrm = \nrm q\nrm$ for $e\in\cU_1$, see \eqref{unitprod}, one  notes that
  \begin{equation} \label{lmii}
   \nrm q - \Pi_n(q)\nrm\ =\ \ds\left\nrm e_n\,e_{n-1}\,\dots\,e_1\,q_n - r_0\right\nrm\,.
  \end{equation}
   Denoting next $q_n = q$ and $q_{n-(i+1)} = e_{i+1}\,\chi_{n-i}(q_{n-i})$
   as in Definition~\ref{def:proj}, we write a telescopic decomposition
   \[
   \begin{aligned}
   q_0 &= e_n \chi_1(q_1) = e_n \, q_1 + e_n\, (\chi_1(q_1) - q_1) = \cdots
  \\
  & = e_n\,e_{n-1}\,\dots\,e_1\,q_n 
    + \sum_{i = 0}^{n-1} e_n\,e_{n-1}\,\dots\,e_{i+1}\,(\chi_{n-i}(q_{n-i}) - q_{n-i})
  \end{aligned}
\]
rewritten as
$$
e_n\,e_{n-1}\,\dots\,e_1\,q_n -r_0  = -  
\left( \sum_{i = 0}^{n-1} e_n\,e_{n-1}\,\dots\,e_{i+1}\,(\chi_{n-i}(q_{n-i}) - q_{n-i})
\right)
+\left(  q_0 -r_0\right).
$$
  The identity (\ref{lmii}) and the triangular inequality   yield
  $$
    \nrm q - \Pi_n(q)\nrm  
    \leq \ds\sum_{i = 0}^{n-1} \left\nrm e_n\,e_{n-1}\,\dots\,e_{i+1}\,(q_{n-i} - \chi_{n-i}(q_{n-i}))\right\nrm\ +\ \left\nrm q_0 - r_0\right\nrm. 
    $$
  Using again \eqref{unitprod} and the fact that $\bar{r}_0 q_0 = M(q_0)^{1/2}$
  (still from Definition~\ref{def:proj}), one gets 
  \[
  \begin{array}{rcl}
  \nrm q - \Pi_n(q)\nrm  
   &\leq& \ds\sum_{i = 0}^{n-1} \left\nrm q_{n-i} - \chi_{n-i}(q_{n-i})\right\nrm\ +\ \left\nrm \overline{r_0}q_0 - (1,0,0,0)\right\nrm\\[1em]
   &\leq& \ds C\,\sum_{i = 0}^{n-1} \nrm q_{n-i}\nrm^{3/2}\,\|\eps(q_{n-i})\|^{1/4}\ +\ |M(q_0)^{1/2}-1|
  \end{array}
  \]
  where the last inequality uses \eqref{e:estim_chi2}.
  Since the correction functions $\chi_{n-i}$ are non-increasing in the $\nrm\cdot\nrm$ norm 
  (see again Proposition~\ref{prop:prop_chi_all}),
  one has $\nrm q_{n-i}\nrm\leq H$ and thus
  \[
   \nrm q - \Pi_n(q)\nrm\ \leq\  C(H)\,\sum_{i = 0}^{n-1}\|\eps(q_{n-i})\|^{1/4}\ +\ C\,\|\eps(q_0)\|^{1/2}\,
  \]
  where we have also used that $|\sqrt{M}-1|\leq\sqrt{|M-1|}$ for all $M \ge 0$.
  Using  the morphism property $M(e_i \hq) = M(\hq)$
  and Estimate \eqref{e:estim_chi3}, we finally write 
  \[
  \|\eps(q_{n-i})\| 
    = \|\eps(\chi_{n-i+1}(q_{n-i+1}))\|
    \leq C(H) \|\eps(q_{n-i+1})\|^{1/4}
    \leq \cdots 
    \leq C(H) \|\eps(q)\|^{1/4^i} 
  \]
  which, combined with the previous estimate,
  yields
  \[
    \nrm q - \Pi_n(q)\nrm\ \leq\ C(H) \,\sum_{i = 0}^{n-1}\|\eps(q)\|^{1/4^{i}}\ +\  C(H)\|\eps(q)\|^{1/4^{n+1/2}}\,.
  \]
 This is enough to conclude.
  \end{proof}

  \section{Numerical illustration}\label{s:implement}

 To illustrate the properties of our projection algorithm we have implemented a global polynomial 
 approximation method. Given some data $(x_r, y_r)_{r = 1,\dots,2n+1}$, our method builds a 
 polynomial with two bounds, $\tilde{p}\in U_{2n}$, such that the values $(\tilde{p}(x_r))_{r}$ 
 are a good approximation to $(y_r)_{r}$.
 For this purpose we begin by interpolating the data $(x_r, y_r)_{r}$ by their Lagrange polynomial 
 $p \in P_{2n}$, and use $p$ as an effective target function. Ne note that in general $p$ may be 
 outside of the desired bounds.
 
 The method is divided in three stages.
 \begin{itemize}[leftmargin = 0in]
  \item In the first stage, one computes a polynomial approximation with one lower bound, 
  $p_0 = a^2+b^2w \in P_{2n}^+$. The goal is to compute explicitly $a$, $b$ and not just $p_0$. Several methods related to this problem have been proposed by the authors in previous contributions \cite{despres_2018_iterative, campos_2019_algorithms}. Here, we use another technique described in Appendix~\ref{s:appendix}.
  \item In the second stage we apply the same method as in the first stage to the data 
  $(x_r, 1-y_r)_{r = 1,\dots,2n+1}$. This yields another polynomial $1-p_1 = c^2+d^2w\in P_{2n}^+$
  and hence a second approximation $p_1$ to the data, now with one upper bound.
  
  \item The third stage consists of applying the projection algorithm defined in Section~\ref{s:projection}. 
  From the polynomials $(a,b,c,d)$ this computes $(\tilde{a},\tilde{b},\tilde{c},\tilde{d}) = \Pi_n(a,b,c,d)$ and provides a polynomial approximation with two bounds $\tilde{p} = \tilde{a}^2 + \tilde{b}^2w$, as described in Corollary~\ref{t:error2}.  The minimization of the dual convex problem which is necessary to compute $\chi_n$ is performed with a Newton conjugate gradient trust-region algorithm \cite{nocedal_2006_numerical}. In our tests, the minimum is reached between $2$ and $5$ iterations. This operation is repeated $n$ times (see Definition~\ref{def:proj}). The cost of one iteration does not depend on $n$.
 \end{itemize}

 In the following, we take $n=5$ so that we are looking for approximations of degree $10$. 
 On the horizontal axis the values correspond to Chebyshev nodes, $x_r = 0.0051$,  $0.0452$,  $0.1221$,   $0.2297$,  $0.3591$, $0.5000$,   $0.6409$,  $0.7703$,  $0.8779$,  $0.9548$, $0.9949$. 

 In the first three test cases we choose different values of $(y_r)_r$, so that the corresponding Lagrange polynomials $p$ have larger amplitudes and exceed the desired bounds. The goal here is to compare 
 qualitatively $p$ with the projected polynomial $\tilde p$, in order to witness the quality of the projection. 
 The last test case is an experimental error analysis. We project a series of polynomials at given distances from the set $U_{2n}$ and compare the numerical convergence rate with the theoretical result of Theorem~\ref{t:error}.

  \subsection{First test case}
  
    In this first test case we choose $y_r=0.1500$,     $0.2402$,    $0.1101$,    $0.0997$,    $0.9062$,    $0.5877$,    $0.5548$,    $0.1095$,    $0.8883$,    $0.6343$ and $0.3360$. Althougt $y_r \in (0,1)$, 
    the Lagrange polynomial $p$ may not be within the bounds. Indeed it exceeds the bounds for 
    $x\approx 0.2$ and $x\approx 0.9$. The results are displayed in Figure~\ref{fig:1}.
    One observes that as expected $p_0\geq 0$, $p_1\leq 1$. Finally the projected polynomial is truly between $0$ and $1$, and seems to be a satisfactory approximation of $p$.

\begin{figure}[h!]
  \begin{center}
  \input{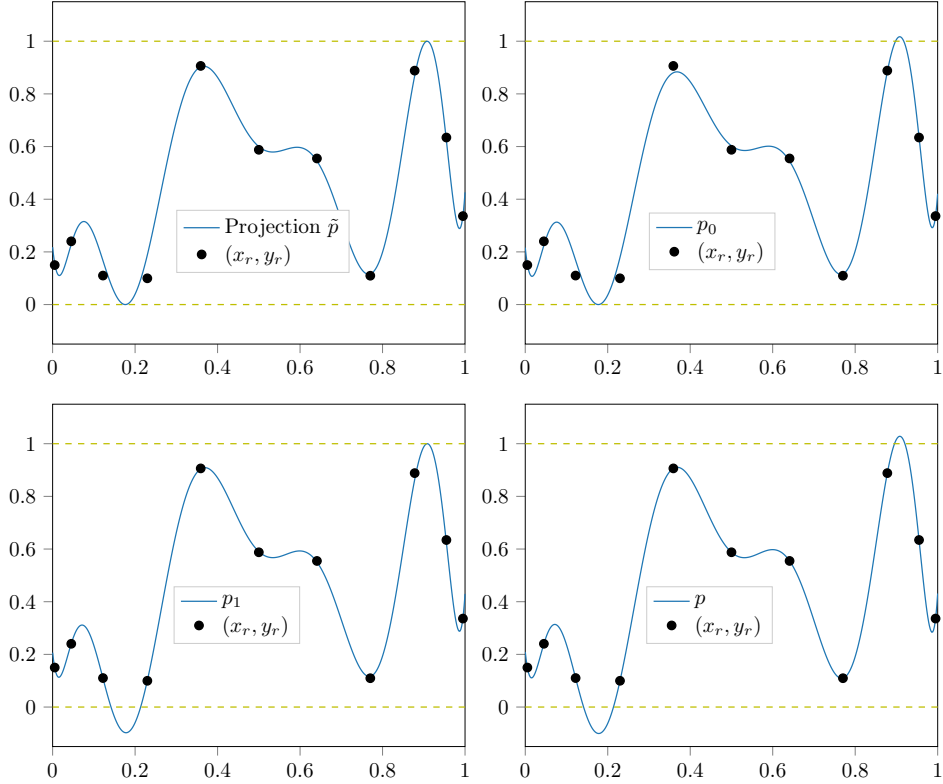}
  \end{center}
\caption{\textbf{First test case:} Bottom right: Lagrange polynomial $p$; Bottom left: Upper bound Lukacs approximation; Top right: Lower bound Lukacs approximation; Top left: Projection $\tilde{p}$; Even if the polynomial
$p_0$ and $p_1$ are marginally out of bounds, a perfect satisfaction of the bounds is observed for $\tilde{p}$.}
     \label{fig:1}
\end{figure}

  \subsection{Second  test case}
  
  Now $y_r=0.3326$,    $0.5950$,   $-0.0938$,   $-0.1245$,    $0.5431$,    $0.8908$,    $1.1076$,   $-0.0181$,    $0.5964$,    $0.4571$ and $-0.1833$.
  The results are displayed on Figure~\ref{fig:2}.  
  Despite the large overshoot and undershoot of $p_0$ and $p_1$ respectively, one sees that the projected polynomial yields a satisfactory approximation of $p$.
\begin{figure}[h!]
  \begin{center}
  \input{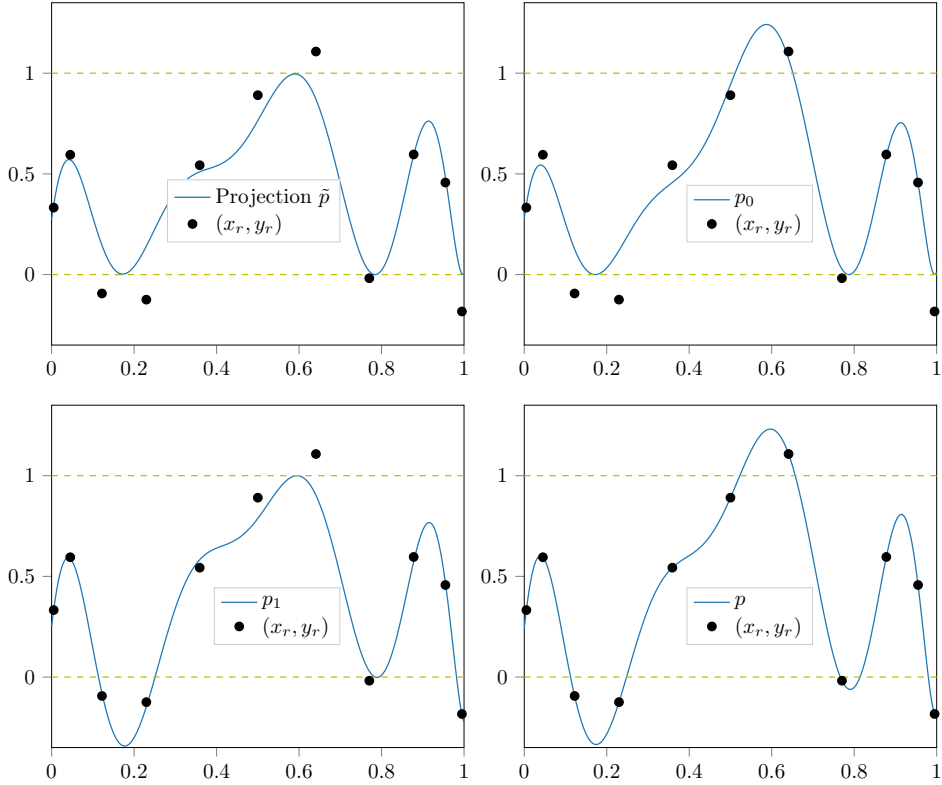}
  \end{center}
\caption{\textbf{Second test case:} Bottom right: Lagrange polynomial $p$; Bottom left: Upper bound Lukacs approximation; Top right: Lower bound Lukacs approximation; Top left: Projection $\tilde{p}$; Even if the polynomial
$p_0$ and $p_1$ are slightly out of bounds, a perfect satisfaction of the bounds is observed for $\tilde{p}$.}
     \label{fig:2}
\end{figure}

\subsection{Third test case}\label{s:test3}

In this third test case $y_r=0.0114$,    $-0.5135$,    $1.3829$,   $-0.0664$, $0.5856$,   $-0.5031$,    $0.8059$,   $-0.2111$,    $0.9622$,    $1.0676$ and $1.2445$. This is a much more severe test in terms of accuracy since the violation of the upper and lower bounds are extreme, and indeed of similar amplitude than the bounds themselves. However we observe in Figure~\ref{fig:3} a perfect behavior in terms of satisfaction of the bounds for the projected polynomial, moreover the qualitative profile of the curve seems to be preserved.

\begin{figure}[h!]
  \begin{center}
  \input{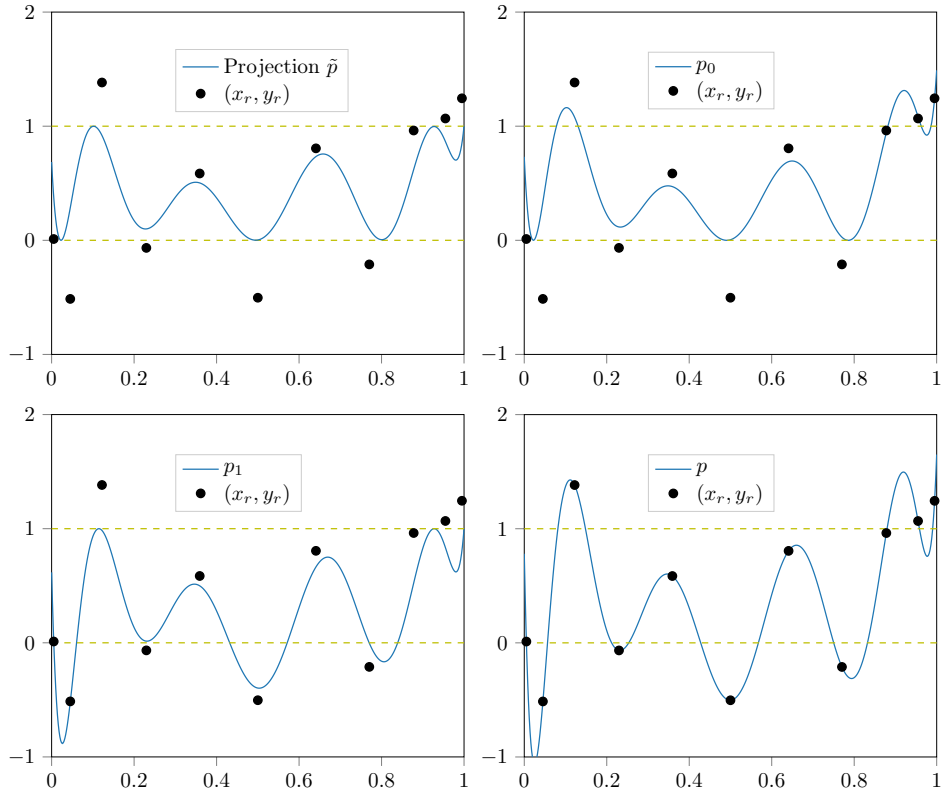}
  \end{center}
\caption{\textbf{Third test case:} Bottom right: Lagrange polynomial $p$; Bottom left: Upper bound Lukacs approximation; Top right: Lower bound Lukacs approximation; Top left: Projection $\tilde{p}$; 
Even if the polynomials $p_0$ and $p_1$ are largely out of bounds, a perfect satisfaction of the bounds is observed for $\tilde{p}$.}
     \label{fig:3}
\end{figure}

\subsection{Fourth test case: error analysis}

In this last numerical test, we want to discuss the estimate of Theorem~\ref{t:error} numerically. To proceed we start by defining 
\[
y_r(t) = t(y_r - \bar{y}) + \bar{y}\,,
\]
where the values $y_r$ are those of the previous test case (Section~\ref{s:test3}), 
$\bar{y}$ is their average and $t\in[0,1]$. From $x_r$ and $y_r(t)$ we  define 
the associated Lagrange polynomial $p_t$. Clearly $p_t = tp + (1-t)\bar{y}$ with $p$ 
the Lagrange polynomial  associated with $(x_r,y_r)_r$. Thus, since $\bar{y}\in[0,1]$ 
and $p\notin[0,1]$ (see Figure~\ref{fig:3}), there is some $t_*\in(0,1)$ such that 
$p^{(t)}\in[0,1]$ if $t\leq t_*$. Above the critical value of $t$ the polynomial 
$p_t$ violates the bounds. We denote by $q_t$ the quaternion corresponding to the 
Lukacs approximations of $p_t$ and we compare $\|\varepsilon(q_t)\|$ and 
$\nrm q_t - \Pi_n(q_t)\nrm$ for various values of $t$. In our numerical test we chose 
$t = 1.$, $0.8$, $0.6$, $0.5$, $0.43$, $0.38$, $0.35$, $0.33$, $0.32$, $0.31$, $0.305$, 
$0.301$, $0.298$, $0.296$, $0.294$, $0.293$, $0.292$, $0.291$, $0.2908$, $0.2906$, $0.2904$ 
and $0.29$. The results are showed on Figure~\ref{fig:4}. The slope is approximately equal 
to $1$ in logarithmic scale which suggests that $\nrm q_t - \Pi_n(q_t)\nrm = O(\|\varepsilon(q_t)\|)$. 
This emphasizes that the error estimate of Theorem~\ref{t:error} is probably far from being sharp. 
Moreover, we see on this test case the convergence and stability of the method.

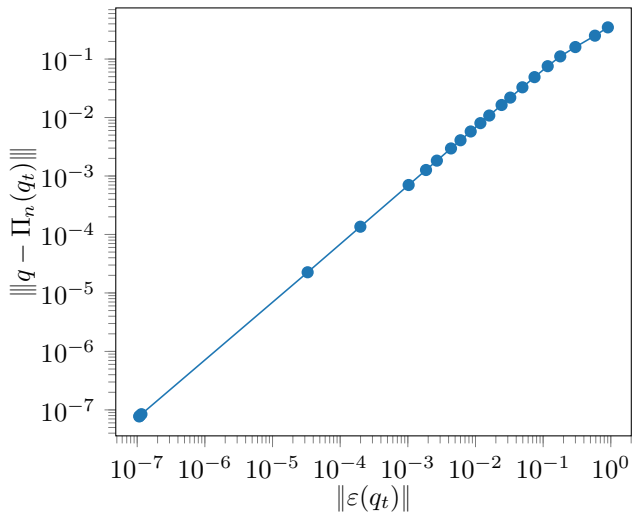
\begin{figure}[h!]
  \begin{center}
\begin{tikzpicture}

\definecolor{color0}{rgb}{0.12156862745098,0.466666666666667,0.705882352941177}

\begin{axis}[
tick align=outside,
tick pos=left,
x grid style={white!69.01960784313725!black},
xlabel={$\|\varepsilon(q_t)\|$},
xmin=4.821374927686e-08, xmax=2.00817567018212,
xmode=log,
y grid style={white!69.01960784313725!black},
ylabel={$\nrm q- \Pi_n(q_t)\nrm$},
ymin=3.59684821893708e-08, ymax=0.749021269852474,
ymode=log,
]
\addplot [semithick, color0, mark=*, mark size=2, mark options={solid}, forget plot]
table [row sep=\\]{%
	0.904596382717595 0.348202419956167 \\
	0.582882054698515 0.251042963384771 \\
	0.299364797580972 0.160341841693937 \\
	0.178949607145395 0.111080974450879 \\
	0.116770639261544 0.0754991266147027 \\
	0.0745730626024416 0.0490673293388529 \\
	0.0493817432944752 0.0328718265873452 \\
	0.0326507457286938 0.0219133021315044 \\
	0.0243064992371092 0.0163822844486174 \\
	0.0159776391832309 0.0108153812335307 \\
	0.0118192791019353 0.00801814678973858 \\
	0.00849562491596383 0.00577363002073382 \\
	0.00600470203122974 0.00408627575056729 \\
	0.00434496678423832 0.00295946725632989 \\
	0.00268595095614338 0.00183112620208017 \\
	0.0018567111511864 0.00126637792942779 \\
	0.00102765975572759 0.000701243954513711 \\
	0.000198789207548508 0.000135721581519144 \\
	3.30415530907898e-05 2.25740107661838e-05 \\
	1.15407071914925e-07 8.44131046727492e-08 \\
	1.08557837212658e-07 7.86033448177595e-08 \\
	1.07033015072621e-07 7.7372116504934e-08 \\
};
\end{axis}

\end{tikzpicture}
  \end{center}
\caption{\textbf{Fourth test case:} Projection error as a function of the 
error polynomial amplitude $\|\varepsilon(q_t)\|$, which measures the distance of the Luk\`acs quaternion $q_t$
to the desired lower and upper bounds. 
Here the experimental convergence rate $\approx 1$  (that is the slope of the line) is much better than what is predicted by Theorem~\ref{t:error} which is a slope
$\approx \frac1{2^{11}}$ for $n=5$.
}\label{fig:4}  

\end{figure}

  \bibliographystyle{plain}
 \bibliography{bibli}
  
  \appendix
  \section{An algorithm for positive polynomial approximation}\label{s:appendix}
  
  Here we briefly describe the method used in the numerical tests to compute the positive (or nonnegative)
  polynomial approximations in Lukacs form. The problem is to find two polynomials $a\in P_n$ and $b\in P_{n-1}$ defining a positive polynomial $p_0(x) = a(x)^2 + b(x)^2 w(x)$ such that given some data $(x_r, y_r)_{r = 1,\dots,R}$  (in general with $R = \mathrm{dim}(P_{2n}) = 2n+1$) the images $(p_0(x_r))_{r}$ are a good approximation of $(y_r)_{r}$.
  
  Our algorithm consists in a least-square minimization where $a$ and $b$ are ``oscillating polynomials'' parametrized by their roots. This parametrization is motivated by the method of \cite{campos_2019_algorithms} where a similar technique has been developped and analysed for positive interpolation.
  
  Mathematically the method relies on the following optimization problem. Find
  \[
  (\alpha^*, \beta^*)\ \in\ \mathrm{argmin}_{\alpha\in\RR^{n+1},\beta\in\RR^{n}}J_t(\alpha,\beta)
  \]
  where the objective function is
  \[
  J(\alpha,\beta)\ =\ \sum_{r = 1}^{R}|a[\alpha](x_r)^2 + b[\beta](x_r)^2w(x_r) - y_r|^2\,,
  \]
  with $a$ and $b$ parametrized as follows,
  \[
    a[\alpha](x)\ =\ 2^{n-1}\,\alpha_0\,\prod_{i=1}^n (x-\alpha_i)\,,\qquad
    b[\beta](x)\ =\ 2^{n-1}\,\beta_0\,\prod_{i=1}^{n-1} (x-\beta_i)\,.
  \]
  The factor $2^{n-1}$ is taken so that $\alpha_0$ and $\beta_0$ are of the same order as the other components of $\alpha$ and $\beta$. Then the approximation polynomial $p$ is defined by 
  \[
   p_0(x) = a[\alpha^*](x)^2 + b[\beta^*](x)^2w(x)
  \]
  The optimization problem is nonlinear and non-convex. However, it can be solved efficiently in practice. Indeed, one can compute explicitely both the gradient and hessian of the functional $J$. In the numerical tests of Section~\ref{s:implement}, we used a Newton conjugate gradient trust-region algorithm. The initial couple $(\alpha,\beta)$ is taken to be appropriate roots of Chebychev polynomials. In this way, the initial polynomials $a[\alpha]$ and $b[\beta]$ are proportional to $T_n$ and $U_n$, yielding $a[\alpha]^2(x) + b[\beta]^2(x)w(x)$ being some constant polynomial. In all the cases  of Section~\ref{s:implement}, $n=5$ and the algorithm converges after around $30$ iterations.

\end{document}